\documentclass[reqno]{amsart}
\usepackage{amssymb}
\usepackage{color}
\usepackage{xcolor}
\usepackage[allcolors = blue, colorlinks = true]{hyperref}
\usepackage{mathptmx}
\usepackage{comment}
\usepackage{graphicx}
\usepackage{caption}
\usepackage{subcaption}
\usepackage[graphicx]{realboxes}
\numberwithin{equation}{section}
\usepackage[
  hmarginratio={1:1},     
  vmarginratio={1:1},     
  textwidth=450pt,        
  heightrounded,          
]{geometry}

\newtheorem{theorem}{Theorem}[section]
\newtheorem{lemma}[theorem]{Lemma}
\newtheorem{corollary}[theorem]{Corollary}
\newtheorem{proposition}[theorem]{Proposition}
\theoremstyle{definition}

\theoremstyle{remark} 
\newtheorem{remark}{Remark}
\newcommand{\R}{\mathbb{R}}
\newcommand{\Rn}{\R^n}
\newcommand{\il}{\Delta_{\infty}}
\newcommand{\ilN}{\Delta_{\infty}^N}

\newcommand{\ue}{u^{\epsilon}}
\newcommand{\la}{\langle}
\newcommand{\ra}{\rangle}

\newcommand{\Om}{\Omega}
\newcommand{\ez}{\epsilon}

\newcommand{\loc}{\textnormal{loc}}
\newcommand{\abs}[1]{\left|#1\right|}
\newcommand{\pro}[1]{\left<#1\right>}

\newcommand{\divt}{\operatorname{div}}

\newcommand{\eps}{{\varepsilon}}
\DeclareMathOperator{\diverg}{div\,}
\begin{document}
\title{Second order Sobolev regularity results for the generalized $p$-parabolic equation}
\author{Yawen Feng}
\address[Yawen Feng]{School of Mathematical Science, Beihang University, Changping District Shahe Higher Education Park South Third Street No. 9, Beijing 102206, P. R. China; and}
\address{Department of Mathematics and Statistics, University of
Jyv\"askyl\"a, PO~Box~35, FI-40014 Jyv\"askyl\"a, Finland}
\email{yawen.y.feng@jyu.fi}

\author{Mikko Parviainen}
\address[Mikko Parviainen]{Department of Mathematics and Statistics, University of
Jyv\"askyl\"a, PO~Box~35, FI-40014 Jyv\"askyl\"a, Finland}
\email{mikko.j.parviainen@jyu.fi}

\author{Saara Sarsa}
\address[Saara Sarsa]{Department of Mathematics and Statistics, University of
Jyv\"askyl\"a, PO~Box~35, FI-40014 Jyv\"askyl\"a, Finland}
\email{saara.m.sarsa@jyu.fi}

\subjclass[2010]{ 35B65, 35K55, 35K65, 35K67.} 
\keywords{General $p$-parabolic equations, viscosity solutions, divergence structures, second order Sobolev regularity, time derivative.} 

\thanks{S. Sarsa was supported by the Emil Aaltonen foundation and the Academy of Finland, project 
354241.}

\begin{abstract}
We study a general class of parabolic equations
\begin{align*}
    u_t-\abs{Du}^\gamma\big(\Delta u+(p-2) \Delta_\infty^N u\big)=0,
\end{align*}  
which can be highly degenerate or singular. 
This class contains as special cases the standard parabolic $p$-Laplace equation and the normalized version that arises from stochastic game theory. Utilizing the systematic approach developed in our previous work we establish second order Sobolev regularity together with a priori estimates and improved range of parameters. In addition we derive second order Sobolev estimate for a nonlinear quantity. This quantity contains many useful special cases. As a corollary we also obtain that a viscosity solution has locally $L^2$-integrable Sobolev time derivative.
\end{abstract}
\maketitle
\section{Introduction}
We study the second order Sobolev regularity of the general $p$-parabolic equation
\begin{equation}
    \label{eq:general-p-parabolic-equation-intro}
    u_t-\abs{Du}^\gamma\big(\Delta u+(p-2) \Delta_\infty^N u\big)=0,
\end{equation}
where $1<p<\infty$, $-1<\gamma<\infty$ and 
$$\Delta_\infty^N u=\frac{\sum_{i,j=1}^n u_{x_i}u_{x_j}u_{x_ix_j}}{\abs{Du}^2}=\frac{\pro{Du,D^2uDu}}{\abs{Du}^2}=\frac{\Delta_\infty u}{\abs{Du}^2}$$
is the normalized infinity Laplacian assuming at this point that $Du\neq 0$. When $\gamma=0$, the equation is called normalized $p$-parabolic equation and when $\gamma=p-2$, it gives standard $p$-parabolic equation, which can be rewritten in a divergence form. Except above two cases, the equation is neither in a divergence form nor uniformly parabolic  and the solutions will have to be interpreted in the viscosity sense. Recently, second order spacial regularity for the solutions to this $p$-parabolic type equation has at least in special cases attracted attention, see H{\o}eg and Lindqvist \cite{lindqvisth20}, Dong, Peng, Zhang and Zhou \cite{dongpzz20} and the authors \cite{fengps22,fengps23}.

First, in \cite{fengps22}, we studied the $W^{1,2}_\loc$-regularity of the nonlinear spatial gradient $\abs{Du}^{\frac{p-2+s}{2}}Du$ for the $p$-parabolic equation
\begin{equation}\label{eq:p-parabolic equation}
    u_t-\Delta_p u=0,
\end{equation}
where $\Delta_p u:=\divt(\abs{Du}^{p-2}Du)$ is the $p$-Laplacian and obtained the sharp range $s>-1$. One of the key tools is the following fundamental inequality
\begin{equation*}
    \abs{D^2u}^2\geq 2\frac{\abs{D^2uDu}^2}{\abs{Du}^2}+\frac{(\Delta u-\Delta_\infty^N u)^2}{n-1}-\left(\Delta_\infty^{N}u\right)^2
\end{equation*}                                    
from Sarsa \cite{sarsa22}. 
In fact, perhaps surprisingly, sharp regularity estimates for equation \eqref{eq:p-parabolic equation} are obtained even with a weaker version of the fundamental inequality, where only the immediate observation that the second term on the right hand side is positive is needed.

In \cite{fengps23}, we developed a systematic approach for the second order spatial regularity of the viscosity solution to the general equation \eqref{eq:general-p-parabolic-equation-intro}. However, the range obtained there for the parameters is obviously non-optimal.  
The purpose of this paper is to improve the range of $p$ for the second order spacial regularity and further the range of $s$ for the $W^{1,2}_\loc$-regularity of the nonlinear quantity $\abs{Du}^{\frac{p-2+s}{2}}Du$.

The idea of the proof is carefully explained below in Section \ref{sec:steps} but we already mention some of the problems.
The first  difficulty in dealing with the general equation instead of the $p$-parabolic equation stems from the fact that this approach gives rise to the mixed terms of the type
\begin{align*}
   |Du|^{-\gamma}u_t\ilN u
\end{align*}
which are difficult to handle. 

The second difficulty is caused by the fact that $u$ is not smooth a priori, and the estimates we would like to write down contain of course second Sobolev derivatives. Also negative powers of the gradient are problematic as the gradient might vanish. A natural approach to tackle these problems is to regularize the equation by adding a small regularization parameter, which removes the singularity. Unfortunately, this causes mismatch with certain terms for example in the fundamental inequality, and adds to the complication of the argument.

Our first main theorem improves the known range of parameters for the second order Sobolev estimate. 
\begin{theorem}\label{thm:W22-more-p}
    Let $n=2$ and $u:\Omega_T\to\R$ be a viscosity solution to the general $p$-parabolic equation \eqref{eq:general-p-parabolic-equation-intro}. If 
    $$3\leq p\leq40 \quad and \quad -1<\gamma<1,$$
    then $D^2u$ exists and belongs to $L^2_\loc(\Omega_T)$. Moreover, we have the estimate
    \begin{equation*}
        \int_{Q_r}\abs{D^2u}^2dxdt\leq\frac{C}{r^2}\left(\int_{Q_{2r}}\abs{Du}^2dxdt+\int_{Q_{2r}}\abs{Du}^{2-\gamma}dxdt\right),
    \end{equation*}
    where $C=C(p,\gamma)>0$ and $Q_r\subset Q_{2r}\Subset\Omega_T$ are concentric parabolic cylinders.
\end{theorem}
If we choose $\gamma=p-2$, then this special case was obtained in \cite{dongpzz20}. To be more precise, they prove that if $1<p<3$, then $D^2u\in L^2_\loc$ for the standard $p$-parabolic equation. Observe that the range $1<p<5$, $-1<\gamma<1$ was already proven in \cite{fengps23}. Thus combining these results the known range of $p$ is now $1<p\le 40$ in this case. 

Our second main result gives a second order Sobolev estimate for a nonlinear quantity. This estimate has been obtained for the elliptic $p$-Laplace equation in \cite{dongpzz20} and the standard $p$-parabolic equation in \cite{fengps22} with a sharp range of parameters. The estimate contains many useful special cases.
\begin{theorem}
    \label{thm:W12-Nonlinear-Gradient-S}
Let $n=2$ and $u:\Omega_T\to\R$ be a viscosity solution to the general $p$-parabolic equation \eqref{eq:general-p-parabolic-equation-intro}. If
\begin{equation*}
    s>\max\{\gamma+1-p, -2-\gamma\}
\end{equation*}
or
\begin{equation*}
    -2-\gamma\geq s>\max\left\{\gamma+1-p, 2p-4-\gamma-2\sqrt{2(p-1)(p-2-\gamma)}\right\}
\end{equation*}
then $D\big(\abs{Du}^{\frac{p-2+s}{2}}Du\big)$ exists and belongs to $L^2_\loc(\Omega_T)$. Moreover, we have the estimate 
\begin{equation} \label{eq:W12-Nonlinear-Gradient-S}
    \begin{aligned}
        \int_{Q_r}\abs{D\big(\abs{Du}^{\frac{p-2+s}{2}}Du\big)}^2dxdt\leq\frac{C}{r^2}\left(\int_{Q_{2r}}\abs{Du}^{p+s}dxdt+\int_{Q_{2r}}\abs{Du}^{p+s-\gamma}dxdt\right),
    \end{aligned}
\end{equation}
where $C=C(p,\gamma,s)>0$ and $Q_{r}\subset Q_{2r}\Subset \Omega_T$ are concentric parabolic cylinders.
\end{theorem}
Observe that even if we obtain the estimate in the first main theorem from the second by selecting $s=2-p$, the ranges are not the same and thus we have stated the theorems separately. 

We briefly discuss the ranges of parameter $s$ in Theorem \ref{thm:W12-Nonlinear-Gradient-S}. 
A simple example in \cite[Example 5.1]{fengps23} shows that the optimal range is $s>\gamma+1-p$. We are not able to prove estimate \eqref{eq:W12-Nonlinear-Gradient-S} for this range, and instead we have two non-optimal lower bounds. In order to see when the optimal lower bound dominates the non-optimal ones, it is illustrative to divide the $(p,\gamma)$-plane into two parts, the upper part where $\gamma+1-p\geq-2-\gamma$, and the lower part where the converse holds.  

In the upper part the optimal lower bound $1+\gamma-p$ dominates the first non-optimal lower bound $-2-\gamma$, and therefore we obtain the optimal range of $s$:
\begin{align*}
    s>\gamma+1-p.  
\end{align*}
In the lower part the second non-optimal lower bound comes into play. A computation shows that in the lower part the second non-optimal lower bound dominates the optimal one. Therefore, in this region combining the ranges, the conditions simplify to 
\begin{align*}
     s>2p-4-\gamma-2\sqrt{2(p-1)(p-2-\gamma)}. 
\end{align*}

We also study the $L^2$-integrability of the time derivative. As observed in Lindqvist \cite{lindqvist08}, which is not clear directly from the definition, a suitable  spatial second order Sobolev regularity together with the equation implies existence as a function and integrability of the time derivative.  
\begin{corollary}[Time derivative] \label{cor:Time-derivative}Let $n=2$ and $u:\Omega_T\to\R$ be a viscosity solution to the general $p$-parabolic equation \eqref{eq:general-p-parabolic-equation-intro}. If $p$ and $\gamma$ satisfy one of the following conditions:\\
(i) $3\leq p\leq40$ and $0\leq\gamma<1$; \\
(ii) $1<p<9\gamma+10$ and $-1<\gamma<\infty$, \\
then the time derivative $u_t$ exists as a function and belongs to $L^2_\loc(\Omega_T)$.
\end{corollary}

Next, we review some known regularity results. When $\gamma=p-2$, equation \eqref{eq:general-p-parabolic-equation-intro} is simply the standard $p$-parabolic equation \eqref{eq:p-parabolic equation}, and the regularity of weak solutions has been studied for example by DiBenedetto in \cite{dibenedetto93}. For example Harnack's inequality and $C^{1,\alpha}_\loc$, $\alpha\in(0,1)$-regularity was obtained there. For the Sobolev regularity, Lindqvist \cite{lindqvist08} obtained the $L^2_\loc$-integrability of the nonlinear quantity $D(\abs{Du}^{\frac{p-2}{2}}Du)$ in the degenerate case $2<p<\infty$, which implies that $D(\abs{Du}^{ p-2} Du)$, $u_t\in L^{\frac{p}{p-1}}_\loc$, see also \cite{lindqvist12}. The singular case $1<p<2$ is studied in \cite{lindqvist17}. For the global regularity, Cianchi and Maz'ya \cite{cianchim19} obtain the estimate for $D(\abs{Du}^{\frac{p-2}{2}}Du)$ with the equation containing the source term $f\in L^2$.

If $\gamma=0$, equation \eqref{eq:general-p-parabolic-equation-intro} gives
 \begin{equation}\label{eq:normalized-p-parabolic}      
u_t-\Delta_p^N u=u_t-\big(\Delta u+(p-2)\Delta_{\infty}^N u\big)=0
\end{equation}
which is called normalized or game theoretic $p$-parabolic equation. This equation can be obtained by considering tug-of-war games as shown by Manfredi, Parviainen and Rossi in \cite{manfredipr10}. It has also been applied into image processing by Does in \cite{does11}.

Since the general equation (\ref{eq:general-p-parabolic-equation-intro}) cannot be rewritten in divergence form, we need to consider the regularity of its viscosity solution defined by Ohnuma and Sato \cite{ohnumas97} and Giga \cite{giga06}. Jin and Silvestre \cite{jins17} show that the viscosity solution $u$ to \eqref{eq:normalized-p-parabolic} belongs $C^{1,\alpha}_\loc$ in spatial variables and $C^{0,\frac{1+\alpha}{2}}_\loc$ in time. With a source term, Attouchi and Parviainen \cite{attouchip18} also obtain the $C^{1,\alpha}_\loc$-regularity of the normalized $p$-parabolic equation. For the Hessian estimate, H{\o}eg and Lindqvist \cite{lindqvisth20} prove that the solution $u\in W^{2,2}_\loc$ if $\frac65<p<\frac{14}{5}$ and derive the gradient on time $u_t$ belongs $L^2_\loc$ if $1<p<\frac{14}{5}$. The work \cite{dongpzz20} from Dong, Peng, Zhang and Zhou includes the second order Sobolev regularity for the normalized $p$-parabolic equation (together with elliptic $p$-Laplace equation and standard $p$-parabolic equation): for $1<p<3+\frac{2}{n-2}$, Hessian $D^2u$ and $u_t$ belong to $L^2_\loc$. This result is included by our previous work \cite{fengps23}. Besides, for the higher regularity, Andrade and Santos \cite{andrades22} established the improved Sobolev regularity to equation \eqref{eq:normalized-p-parabolic}. 

For the general equation (\ref{eq:general-p-parabolic-equation-intro}), the spatial $C^{1,\alpha}_\loc$-regularity was proven by Imbert, Jin and Silvestre \cite{imbertjs19}. Attouchi \cite{attouchi20} and Attouchi and Ruosteenoja \cite{attouchir20} proved the spatial gradient $Du$ is H\"{o}lder continuous with a source term respectively in the degenerate and singular case. In \cite{kurkinenps23} Kurkinen, Parviainen and Siltakoski proved the elliptic Harnack's inequality for this equation, whereas  in \cite{parviainenv20} Parviainen and V\'{a}zquez, and in \cite{kurkinens} Kurkinen and Siltakoski obtained different intrinsic Harnack's inequalities.
For the elliptic $p$-Laplace equation, the second order Sobolev regularity was studied by Manfredi and Weitsman \cite{manfredi88} using the Cordes condition \cite{cordes61, maugerips00}.

This paper is organized as follows. Section \ref{Sec:Notation-Definitions-Lemmas} gives out some basic notation and important lemmas as useful tools for the proof. Section \ref{sec:ProofofMain} includes the proof of main theorems. The proof of Corollary \ref{cor:Time-derivative} is also given in Section \ref{sec:ProofofMain}, after the proof of Theorems \ref{thm:W22-more-p} and \ref{thm:W12-Nonlinear-Gradient-S}.

\subsection{Main steps of the proof}
\label{sec:steps}

As pointed out above, we aim at obtaining the estimate
\begin{align*}
  \int_{Q_r}\abs{D\big(\abs{Du}^{\frac{p-2+s}{2}}Du\big)}^2dxdt\leq\frac{C}{r^2}\left(\int_{Q_{2r}}\abs{Du}^{p+s}dxdt+\int_{Q_{2r}}\abs{Du}^{p+s-\gamma}dxdt\right),
\end{align*}
which with $s=2-p$ gives
\begin{align*}
 \int_{Q_r}\abs{D^2u}^2dxdt\leq\frac{C}{r^2}\left(\int_{Q_{2r}}\abs{Du}^2dxdt+\int_{Q_{2r}}\abs{Du}^{2-\gamma}dxdt\right).
\end{align*}
At this point we omit the regularizations in order to avoid excess technicalities, and assume that the solution is smooth with the nonvanishing gradient. There are two slightly different approaches in the literature. In for example \cite{lindqvist08}, \cite{sarsa22} and \cite{fengps22}, the first step in higher order derivative estimate for other equations is to differentiate the equation. For the equation of this paper, the steps would be as follows 
(Step 4 already anticipates the systematic approach used in \cite{fengps23} and in this paper).
\begin{enumerate}
\item[Step 1] Differentiate the equation.
\item[Step 2] We multiply the equation with a suitable test function,
use $2u_{x_k}\partial_t u_{x_k}=\partial_t u^2_{x_k}$ for the time derivative, sum over indexes and manipulate to obtain an intermediate key estimate.
\item[Step 3a] Improvement of range via the fundamental inequality: purpose in both Step 3a and 3b is to adjust the coefficients on the left hand side of the key estimate to get a better range. Use fundamental inequality to estimate the term with $|D^2 u|$ in the key estimate. In this step we also replace a problematic mixed term using the equation.
\item[Step 3b] Improvement of range via the hidden divergence structure: Add terms with hidden divergence structure: their purpose is to adjust the coefficients on the left hand side of the key estimate to get a better range while producing controllable divergence form terms on the right hand side. 
\item[Step 4] Interpret the terms on the left hand side of the modified key estimate as a quadratic form and show that it is positive definite. This suffices for the desired result. 
\end{enumerate}
 
Maybe surprisingly, the differentiation of the equation can be avoided in order to arrive at the same estimates. This approach, based on utilizing a version of the fundamental inequality and hidden divergence structure, has been used for different equations in \cite{dongpzz20}. For the equation of this paper and recalling Step 4 from above and \cite{fengps23}, the steps are as follows.
\begin{enumerate}
\item[Step 1] Add hidden divergence structure identities with coefficients to be determined later.
\item[Step 2] Use fundamental inequality in order to get improved range of parameters.  
\item[Step 3] Interpret the terms on the left hand side of the modified key estimate as a quadratic form and show that it is positive definite.  
 \end{enumerate}
Now we explain these Steps 1-3 more precisely.

\noindent {\bf Step 1}:
As explained in detail in \cite{fengps23}, by differentiating the equation  we could arrive at 
\begin{align*}
    &|Du|^{p-2+s}
    \Big\{|D^2u|^2+(p-2+s)|D|Du||^2+s(p-2)(\ilN u)^2 
    +(p-2-\gamma)|Du|^{-\gamma}u_t\ilN u
    \Big\} \nonumber\\
     &=
    \diverg\big(|Du|^{p-2+s} AD^2u Du\big)
    -\frac{(|Du|^{p+s-\gamma})_t}{p+s-\gamma}\nonumber\\
    &=\diverg\big(|Du|^{p-2+s }(D^2uDu-\Delta uDu)\big)
 +u_t\diverg\big(|Du|^{p-2+s-\gamma}Du\big),
\end{align*}
where $A=I+(p-2)\frac{Du\otimes Du}{\abs{Du}^2}$ is a uniformly positive definite $n\times n$-matrix and $I$ denotes the identity matrix.
The last equality results from the use of equation and a short computation; after another computation done later, one can further see that the right hand side can be written in divergence form. 
This also highlights that one of the problems is the mixed term
\begin{align*}
(p-2-\gamma)\abs{Du}^{-\gamma}u_t\ilN u
\end{align*}
since the nonnegativity of other terms on the left hand side above  seem to be easier to control in a suitable range of parameters.
Using the equation
and the shorthand notation 
\begin{align}
\label{eq:mod-notation}
    \abs{D_T\abs{Du}}^2:=\frac{|D^2uDu|^2}{|Du|^2}-(\ilN u)^2
    \quad\text{and}\quad
    \Delta_T u:=\Delta u-\ilN u,  
\end{align}
 this mixed term can be written as
\begin{align*}
\abs{Du}^{-\gamma}u_t\ilN u=\Delta_Tu\ilN u+(p-1)(\ilN u)^2.
\end{align*} 
Using this and the above shorthand, we get
\begin{equation}\label{eq:key-mod}
\begin{aligned}       
    &|Du|^{p-2+s}
    \Big\{|D^2u|^2+(p-2+s)|D_T|Du||^2+(p-2-\gamma)\Delta_T u\ilN u\\
    &\hspace{10 em}+ \big( (p-2+s)+s(p-2)+(p-1)(p-2-\gamma)\big)(\ilN u)^2 
    \Big\} \\
    &=  |Du|^{p-2+s}
    \Big\{|D^2u|^2+(p-2+s)|D_T|Du||^2+(p-2-\gamma)\Delta_T u\ilN u \\
    &\hspace{10 em}+ \big((p-1)(p-1+s-\gamma)-1\big)(\ilN u)^2 
    \Big\} \\
    &=\diverg\big(|Du|^{p-2+s }(D^2uDu-\Delta uDu)\big)
    +u_t\diverg\big(|Du|^{p-2+s-\gamma}Du\big)
\end{aligned}
\end{equation}
However, this approach is actually not needed, but it motivates using (partly hidden) divergence structures, that is, the terms on the right hand side. For the first term on the right hand side, we have by a computation
 \begin{align*}
    |Du|^{p-2+s}&
    \Big\{|D^2 u|^2-(\Delta u)^2
    +(p-2+s)\frac{|D^2uDu|^2}{|Du|^2}-(p-2+s )\Delta u\ilN u\Big\}\nonumber\\
     &=\diverg\big(|Du|^{p-2+s}(D^2uDu-\Delta uDu)\big),
\end{align*}
which holds for any smooth function with nonvanishing gradient. Using (\ref{eq:mod-notation}) this gives
\begin{equation} \label{eq:div1}
\begin{aligned}
    |Du|^{p-2+s}
    &\Big\{|D^2 u|^2+(p-2+s)\abs{D_T\abs{Du}}^2-(\Delta_T u)^2-(p+s )\Delta_T u\ilN u-(\ilN u)^2
    \Big\} \\
     &=\diverg\big(|Du|^{p-2+s}(D^2uDu-\Delta uDu)\big),
\end{aligned}
\end{equation}
The second term on the right of (\ref{eq:key-mod}) can be written by using 
 \begin{align*}
u_t=|Du|^{\gamma}\big(\Delta_Tu+(p-1)\ilN u\big) \quad \text{and}\quad
\Delta_{p+s-\gamma}^Nu=\Delta_Tu+(p-1+s-\gamma)\ilN u 
\end{align*}
as
\begin{align*}
u_t&\diverg\big(|Du|^{p-2+s-\gamma}Du\big)\\
=&|Du|^{\gamma}\big(\Delta_Tu+(p-1)\ilN u\big)\cdot|Du|^{p-2+s-\gamma}\cdot\big(\Delta_Tu+(p-1+s-\gamma)\ilN u \big)  \\
 =&|Du|^{p-2+s }\Big\{(\Delta_Tu)^2+(2 p-2+s-\gamma) \Delta_Tu\ilN u+(p-1)(p-1+s-\gamma)(\ilN u)^2\Big\}. 
\end{align*}
Moreover, it is of divergence form, since for any smooth function a short computation shows
\begin{align*}
u_t\diverg\big(|Du|^{p-2+s-\gamma}Du\big)
    &=
    \diverg(u_t|Du|^{p-2+s-\gamma}Du)-\frac{(|Du|^{p+s-\gamma})_t}{p+s-\gamma}.
\end{align*}
Combining the above equalities, we arrive at
\begin{align}
\label{eq:div2}
|Du|^{p-2+s}&\Big\{(\Delta_Tu)^2+(2 p-2+s-\gamma) \Delta_Tu\ilN u+(p-1)(p-1+s-\gamma)(\ilN u)^2\Big\}\nonumber \\
&=\diverg(u_t|Du|^{p-2+s-\gamma}Du)-\frac{(|Du|^{p+s-\gamma})_t}{p+s-\gamma}
\end{align}
Adding the divergence structures (\ref{eq:div1}) and (\ref{eq:div2}) with suitable nonnegative coefficients $w_1$ and $w_2$, we obtain
\begin{align}
\label{eq:weighted-ineq1}
    &|Du|^{p-2+s }
    \Big\{w_1\abs{D^2 u}^2+w_1(p-2+s)\abs{D_T\abs{Du}}^2+  (w_2-w_1) (\Delta_T u)^2 
    \\
    &\quad+\big(w_2(2p-2+s-\gamma)-w_1(p +s )\big)\Delta_Tu\ilN u
    +\big(w_2(p-1)(p-1+s-\gamma)-w_1\big)(\ilN u)^2\Big\}\nonumber \\
    &=
    w_1\diverg\big(|Du|^{p-2+s }(D^2uDu-\Delta uDu)\big)
    +w_2 \big(\diverg(u_t|Du|^{p-2+s-\gamma}Du)-\frac{(|Du|^{p+s-\gamma})_t}{p+s-\gamma} \big).\nonumber
\end{align}
Below we use a shorthand `RHS div-terms' for the right hand side. Also observe that choosing $w_1=w_2=1$, this reduces to (\ref{eq:key-mod}).
Now we would like to get something like
\begin{align*}
    |Du|^{p-2+s}
    |D^2u|^2
    \lesssim \text{RHS div-terms}.
\end{align*}
This could then be integrated by parts against a test function, and after using Young's inequality on the terms that contain $D^2 u$ on the right, we could embed them into the left to obtain the final result. For this we need to estimate the excess terms on the left hand side of \eqref{eq:weighted-ineq1}.

\noindent {\bf Step 2}:
Next we write the fundamental inequality in the form 
\begin{align}
\label{eq:modified-fund}
    2|D_T|Du||^2+\frac{(\Delta_T u)^2}{n-1}+(\ilN u)^2\le|D^2u|^2.
\end{align}
Using this in (\ref{eq:weighted-ineq1}) we get
\begin{equation}\label{eq:mod-key}
   \begin{aligned}
    &|Du|^{p-2+s } \Big\{
    w_1(p+s)\abs{D_T\abs{Du}}^2+ \Big(w_2-\frac{n-2}{n-1}w_1\Big)(\Delta_T u)^2 \\
    &+\big(w_2(2p-2+s-\gamma)-w_1(p +s )\big)\Delta_Tu\ilN u
    +w_2(p-1)(p-1+s-\gamma)(\ilN u)^2\Big\}  \\
    &\le \text{RHS div-terms}. 
    \end{aligned} 
\end{equation}

We need that the terms together on the left hand side are strictly positive: indeed we later prove that this suffices for the main estimate but at this point one could think that for a small $\lambda>0$ we could write  in (\ref{eq:weighted-ineq1})
\begin{align*}
|D^2u|^2=(1-\lambda)|D^2u|^2+\lambda|D^2u|^2
\end{align*}
and use (\ref{eq:modified-fund}) only for $(1-\lambda)|D^2u|^2$. Thus if we knew the positivity for the left hand side of (\ref{eq:mod-key}) (or for slightly modified coefficients because there is $1-\lambda$ in play), we would get the desired estimate.

\noindent {\bf Step 3}: 
Now we observe that the last three terms on the left hand side of (\ref{eq:mod-key}) can also be written as a quadratic form
$$ Q=\la\bar{x},M\bar{x}\ra,$$
where $\bar{x}=(\Delta_T u,\ilN u)^T\in \R^2$ is a vector and
$$ M=
\begin{bmatrix}
w_2- \displaystyle{\frac{n-2}{n-1}} w_1 & \displaystyle{\frac{1}{2}}\big(w_2(2p-2+s-\gamma)-w_1(p+s)\big) \\
\displaystyle{\frac{1}{2}}\big(w_2(2p-2+s-\gamma)-w_1(p+s)\big) & w_2(p-1)(p-1+s-\gamma)
\end{bmatrix}.
$$
As we stated above, showing that left hand side is strictly positive suffices for the main result. Thus we have reduced the problem to finding suitable weights $w_1,w_2$ and showing  that $M$ is positive definite when $\gamma, p$ and $s$ satisfy the desired range condition. This is just an algebraic problem even if not always technically easy.  

This procedure gives in the smooth case, as explained in \cite{fengps23}, the range
\begin{align*}
      s>\max\left\{\gamma+1-p, -1-\frac{p-1}{n-1}\right\}.  
\end{align*}
In the general case without smoothness assumption, this paper improves the known range but the optimal range remains open.
\section{Preliminaries and auxiliary lemmas}\label{Sec:Notation-Definitions-Lemmas}
In this section, we will first introduce some notation used throughout this paper. Then we recall three important parts in \cite{fengps23}: two good divergence structures, the key estimate and some auxiliary lemmas, which also are available in this paper for simplifying the proof.

Let $\Omega\subset\Rn$, $n\geq2$, be a bounded domain. We denote the parabolic cylinder by
$$ \Omega_T:=\Omega\times(0,T) $$
and its parabolic boundary by 
$$ \partial_p\Omega_T:=(\Omega\times\{0\})\cup(\partial\Omega\times[0,T]). $$ 
Moreover, the cylinder 
$$Q_r(x_0,t_0):=B_r(x_0)\times[t_0,t_0+r^2)$$
will also be denoted by $Q_r$ if no confusion arises. If $\overline{U}\subset\Omega$, we write $U\Subset \Omega$.

In this paper, we study the viscosity solution to the general $p$-parabolic equation
\begin{equation}\label{eq:general-p-parabolic-equation}
    u_t-\abs{Du}^\gamma\Delta_p^Nu=0,
\end{equation}
where
$$\Delta_p^N u=\Delta u+(p-2)\Delta_\infty^N u=\abs{Du}^{2-p}\divt(\abs{Du}^{p-2}Du)$$
is the normalized $p$-Laplacian with $1<p<\infty$ and $-1<\gamma<\infty$.

Since the operator in equation \eqref{eq:general-p-parabolic-equation} may be singular, we need to take good care of the definition of suitable viscosity solution, which can be found in \cite{fengps23, ohnumas97}.
However, here we mainly work with smooth solutions $\ue:\Omega_T\to\R$ to the regularized equation
\begin{equation}\label{eq:regularized-PDE}
    \ue_t=(\abs{D\ue}^2+\ez)^{\frac{\gamma}{2}}\bigg(\Delta\ue+(p-2)\frac{\il\ue}{\abs{D\ue}^2+\ez}\bigg)
\end{equation}
of the original general $p$-parabolic equation.
This allows us to use the derivatives and second derivatives in the intermediate steps. At the end we pass to the limit $\eps\to 0$ and obtain a viscosity solution.

\subsection{Nonlinear divergence structures}
In this subsection, we will recall some useful inequalities about generic smooth functions. Their role in the proof was explained in Section \ref{sec:steps}.

By Rademacher's theorem, the local Lipschitz continuity of the norm of gradient $\abs{Du}$ implies that $D\abs{Du}$ exists almost everywhere in space. Moreover, if $(x_0,t_0)\in\Omega_T$ is a space-time point where $\abs{Du}$ is differentiable and $Du(x_0,t_0)=0$, then $D\abs{Du}(x_0,t_0)=0$. 
By a suitable orthonormal basis of $\Rn$, for the points where $D\abs{Du}$ exists, we can define
\begin{equation}\label{eq:Orthogonal-second-order-quantity}
    \abs{D_T\abs{Du}}^2:=\abs{D\abs{Du}}^2-(\Delta_\infty^N u)^2\geq0
\end{equation}
and
\begin{equation}\label{eq:Orthogonal-Laplacian}
     \Delta_Tu:=\Delta u-\Delta_\infty^N u
\end{equation} 
almost everywhere in space in $\Omega_T$. These observations and notations will be utilized below.

\begin{lemma}[Fundamental equality in plane]\label{lem:Fundamental-equality}
Let $u:\Omega_T\to \R$ be a smooth function in the plane. Then
\begin{equation}\label{eq:Fundamental-equality}
    \abs{D^2u}^2=2\abs{D_T\abs{Du}}^2+(\Delta_Tu)^2+(\Delta_\infty^Nu)^2
\end{equation}
almost everywhere in space in $\Omega_T$.
\end{lemma}
For the proof of Lemma \ref{lem:Fundamental-equality}, we refer to \cite{sarsa22,fengps22}, which also prove a fundamental inequality for higher dimensions $n\geq3$.

The following lemmas show that certain terms that first appear to be non-divergence form, can actually be expressed in a divergence form.
The proofs of these lemmas are direct calculations, see \cite{fengps23} for details.

\begin{lemma}[Hidden divergence structure 1]
\label{lem:regularized-divergence-structure-1}
Let $u:\Omega_T\to\R$ be a smooth function. Then for any $\alpha\in\R$ and $\ez>0$, one has
\begin{equation*}
    \begin{aligned}
        &(\abs{Du}^2+\ez)^{\frac{\alpha}{2}}\bigg(\abs{D^2u}^2-(\Delta u)^2+\alpha\frac{\abs{D^2uDu}^2}{\abs{Du}^2+\ez}-\alpha\Delta u\frac{\Delta_\infty u}{\abs{Du}^2+\ez}\bigg)\\
        =&\divt\big((\abs{Du}^2+\ez)^{\frac{\alpha}{2}}(D^2uDu-\Delta uDu)\big).
    \end{aligned}
\end{equation*}
\end{lemma}

\begin{lemma}[Hidden divergence structure 2]
\label{lem:regularized-divergence-structure-2}
Let $u:\Omega_T\to\R$ be a smooth function. Then for any $\beta\in\R$ and $\ez>0$, one has
\begin{equation*}
        \begin{aligned}
            &u_t(\abs{Du}^2+\ez)^{\frac{\beta}{2}}\Big(\Delta u
            +\beta\frac{\il u}{\abs{Du}^2+\ez}\Big)\\
            =&u_t\divt\big((\abs{Du}^2+\ez)^{\frac{\beta}{2}}Du\big)\\
            =&\begin{cases}
                \divt\big(u_t(\abs{Du}^2+\ez)^{\frac{\beta}{2}}Du\big)-\frac{1}{\beta+2}\big((\abs{Du}^2+\ez)^{\frac{\beta+2}{2}}\big)_t      & \text{if }\beta\neq-2,\\
                \divt\big(u_t(\abs{Du}^2+\ez)^{-1}Du\big)-\frac12\big(\ln(\abs{Du}^2+\ez)\big)_t      & \text{if }\beta=-2.
            \end{cases}
        \end{aligned}
\end{equation*}
\end{lemma}
For $\alpha\in\R$ and $\epsilon>0$, we denote the `first good divergence structure' as
\begin{equation}\label{eq:Good-Divergence-Structure-1}
    GD_1(\alpha):=\divt\big((\abs{Du}^2+\ez)^{\frac{\alpha}{2}}(D^2uDu-\Delta uDu)\big)
\end{equation}
and the `second good divergence structure'
\begin{equation}\label{eq:Good-Divergence-Structure-2}
\begin{aligned}
    GD_2(\alpha):=
    \begin{cases}
        \divt\big(u_t(\abs{Du}^2+\ez)^{\frac{\alpha-\gamma}{2}}  Du \big)-\frac{\big((\abs{Du}^2+\ez)^{\frac{\alpha-\gamma+2}{2}}\big)_t}{\alpha-\gamma+2}    & \text{if }\alpha\neq\gamma-2,\\
        \divt\big(u_t(\abs{Du}^2+\ez)^{-1}Du\big)-\frac12\big(\ln(\abs{Du}^2+\ez)\big)_t      & \text{if }\alpha=\gamma-2.
    \end{cases}
\end{aligned} 
\end{equation}

\subsection{The key estimate}
In this section, we will recall a key estimate in \cite{fengps23}, which plays an important role in the proof. 
We consider the following weighted sum of above `good divergence structures' \eqref{eq:Good-Divergence-Structure-1} and \eqref{eq:Good-Divergence-Structure-2}:
\begin{equation}\label{eq:S}
 \begin{aligned}
    S:=&w_1GD_1(p-2+s)+w_2GD_2(p-2+s)\\
&+\epsilon w_3GD_1(p-4+s)+\epsilon w_4GD_2(p-4+s) 
\end{aligned}   
\end{equation}
for some parameter $s\in\R$ and some weights $w_1,w_2,w_3,w_4\in\R$. 

Compared to what was explained in Section \ref{sec:steps} under regularity assumptions, in the actual proof we need to regularize the equation. This produces additional terms, and in order to control them the last two $\eps$-terms above were added.  
By using Lemma \ref{lem:regularized-divergence-structure-1} and \ref{lem:regularized-divergence-structure-2}, we rewrite $S$ as a linear combination of time derivatives and second order spatial derivatives quantities. To this end, we first denote 
\begin{equation*}
    \theta:=\frac{\abs{Du}^2}{\abs{Du}^2+\ez}\in[0,1)
    \quad
    \text{and}
    \quad
    \kappa:=1-\theta=\frac{\ez}{\abs{Du}^2+\ez}\in(0,1],
\end{equation*}
thus we have
\begin{equation*}
     \frac{\abs{D^2uDu}^2}{\abs{Du}^2+\ez}=\theta\abs{D\abs{Du}}^2
    \quad
    \text{and}
    \quad
    \frac{\il u}{\abs{Du}+\ez}=\theta\ilN u
\end{equation*}
almost everywhere in space in $\Omega_T$.
Plugging the good divergence structures i.e. Lemma \ref{lem:regularized-divergence-structure-1} (with $\alpha=p-2+s$ and $p-4+s$) and Lemma \ref{lem:regularized-divergence-structure-2} (with $\beta=p-2+s-\gamma$ and $p-4+s-\gamma$) into the definition \eqref{eq:S} of $S$ from above, for any smooth function $u$, one has
\begin{equation*}
    \begin{aligned}
        S=&
        w_1(|Du|^2+\epsilon)^{\frac{p-2+s}{2}}
\bigg\{|D^2u|^2-(\Delta u)^2
    +(p-2+s)\bigg(\frac{|D^2uDu|^2}{|Du|^2+\epsilon}-\Delta u\frac{\il u}{|Du|^2+\epsilon}\bigg)\bigg\}\\
    &+w_2u_t(|Du|^2+\epsilon)^{\frac{p-2+s-\gamma}{2}}
    \Big(\Delta u +(p-2+s-\gamma)\frac{\il u}{|Du|^2+\epsilon}\Big)\\
    &+\ez w_3(|Du|^2+\epsilon)^{\frac{p-4+s}{2}}
\bigg\{|D^2u|^2-(\Delta u)^2
    +(p-4+s)\bigg(\frac{|D^2uDu|^2}{|Du|^2+\epsilon}- \Delta u\frac{\il u}{|Du|^2+\epsilon}\bigg)\bigg\}\\
    &+\ez w_4u_t(|Du|^2+\epsilon)^{\frac{p-4+s-\gamma}{2}}
    \Big(\Delta u +(p-4+s-\gamma))\frac{\il u}{|Du|^2+\epsilon}\Big).
    \end{aligned}
\end{equation*}
Combining similar terms and recalling the existence of the quantity $D\abs{Du}$, we obtain that
\begin{equation*}
\begin{aligned}
    S=&(|Du|^2+\epsilon)^{\frac{p-2+s}{2}}\Big\{c_1\big(\abs{D^2u}^2-(\Delta u)^2\big)+c_2\big(\abs{D\abs{Du}}^2-\Delta u \ilN u\big)\\
    &+c_3(|Du|^2+\epsilon)^{-\frac{\gamma}{2}}u_t\Delta u+c_4(|Du|^2+\epsilon)^{-\frac{\gamma}{2}}u_t\ilN u\Big\}
\end{aligned}
\end{equation*}
almost everywhere in space in $\Omega_T$, where 
\begin{equation}\label{eq:expression-c1-c4}
    \begin{cases}
        c_1=w_1+w_3\kappa,\quad &c_2=\big(w_1(p-2+s)+w_3(p-4+s)\kappa\big)\theta,\\
        c_3=w_2+w_4\kappa,\quad &c_4=\big(w_2(p-2+s-\gamma)+w_4(p-4+s-\gamma)\kappa\big)\theta.
    \end{cases}
\end{equation}
Moreover, by above expressions \eqref{eq:Orthogonal-second-order-quantity} and \eqref{eq:Orthogonal-Laplacian}, $S$ can be rewritten as 
\begin{equation}\label{eq:S-notation-orthogonal}
    \begin{aligned}
        S=&
    (|Du|^2+\epsilon)^{\frac{p-2+s}{2}}\Big\{c_1\big(\abs{D^2u}^2-(\Delta_T u)^2-(\ilN u)^2\big)+c_2 \abs{D_T\abs{Du}}^2\\
    &-(2c_1+c_2)\Delta_T u\ilN u+c_3 (|Du|^2+\epsilon)^{-\frac{\gamma   } {2}}u_t\Delta_T u\\
    &+(c_3+c_4)(|Du|^2+\epsilon)^{-\frac{\gamma}{2} }u_t\ilN u\Big\}
    \end{aligned}
\end{equation}
almost everywhere in place of $\Omega_T$. According to the regularized equation \eqref{eq:regularized-PDE}, we replace the time derivatives $u_t$ in above expression  \eqref{eq:S-notation-orthogonal} with the spacial derivatives. Then the key estimate (or equality) for a smooth solution $u$ to the regularized equation 
\begin{equation}
    \begin{aligned}\label{eq:key-estimate}
    &(\abs{Du}^2+\epsilon)^{\frac{p-2+s}{2}}\Big\{c_1\abs{D^2u}^2+c_2 \abs{D_T\abs{Du}}^2+(c_3-c_1)(\Delta_T u)^2\\
    +&\big((c_3+c_4)P_\theta-c_1\big)(\ilN u)^2
    +\big(c_3P_\theta+(c_3+c_4)-(2c_1+c_2)\big)\Delta_T u\ilN u\Big\}
    =S
    \end{aligned}
\end{equation}
is achieved, where we denote
\begin{equation}\label{def:P_theta}
    P_\theta:=(p-2)\theta+1\in(0,\infty)
\end{equation}
for the sake of brevity. This should be compared to (\ref{eq:weighted-ineq1}) when $w_3=w_4=0$ and $\eps=0$ implying $\kappa=0,\ \theta=1$ and $P_\theta=p-1$. 
Next replacing $\abs{D^2u}^2$ in \eqref{eq:key-estimate} using the fundamental equality 
\eqref{eq:Fundamental-equality}, we get
\begin{equation}\label{eq:S-expression-general}
    \begin{aligned}
     (\abs{Du}^2+\epsilon)^{\frac{p-2+s}{2}}\left\{(2c_1+c_2) \abs{D_T\abs{Du}}^2+Q\right\}=S,
    \end{aligned}
\end{equation}
where
$$Q=c_3(\Delta_T u)^2+(c_3+c_4)P_\theta(\ilN u)^2+\big(c_3P_\theta+(c_3+c_4)-(2c_1+c_2)\big)\Delta_T u\ilN u$$
is a quadratic form in $\Delta_T u$ and $\ilN u$. 
Moreover, we can  rewrite $Q$ into the compact form
\begin{align*}
    Q=\left<\bar{x},M\bar{x}\right>,
\end{align*}
where 
$$
\bar{x}=(\Delta_T u,\ilN u)^T\in\R^2
$$
is a vector and
$$M:=
\begin{bmatrix}
    c_3  &\displaystyle\frac12 \big(c_3P_\theta+(c_3+c_4)-(2c_1+c_2)\big)\\
    \displaystyle\frac12 \big(c_3P_\theta+(c_3+c_4)-(2c_1+c_2)\big) & (c_3+c_4)P_\theta
\end{bmatrix}
\in\R^{2\times2}$$
is a symmetric matrix.

\subsection{Auxiliary lemmas}\label{subsec:Sum S}
In this section, we recall auxiliary lemmas from \cite{fengps23}. Lemma \ref{lem:positive-terms-away} states that if we are able to show that the (after modification with the fundamental (in)equality) auxiliary terms in (\ref{eq:key-estimate}) form a positive definite quadratic form, then we can estimate these auxiliary terms away and get estimate for the desired quantity containing the second derivatives. In Lemma \ref{lem:lemma-for-final-estimate}, we get an integral estimate by multiplying with a smooth test function, integrating, and integrating by parts the divergence form terms on the right.    

\begin{lemma}\label{lem:positive-terms-away}
Let $\ue:\Omega_T\to\R$ be a smooth solution to \eqref{eq:regularized-PDE}, $S$ as in \eqref{eq:S}, $c_1$ as in \eqref{eq:expression-c1-c4}, and $\ez>0$. Suppose that we can select $w_1,w_2,w_3,w_4\in\R$ such that $c_1 >0$, $c >0$
and
\begin{equation*}
    (\abs{D\ue}^2+\ez)^{\frac{p-2+s}{2}}\big(c\abs{D_T\abs{D\ue}}^2+Q\big)\leq S
\end{equation*}
almost everywhere in space in $\Omega_T$, where 
$$Q=\pro{\bar{x},M\bar{x}}$$
with $\bar{x}=(\Delta_T\ue, \ilN \ue)^T\in\R^2$ and a uniformly bounded positive definite (with a uniform constant) symmetric matrix $M\in\R^{2\times2}$.
Here $c_1,c$ and $M$ only depend on $n,p,\gamma,s,w_1,w_2,w_3$ and $w_4$. Then there exists a $\lambda=\lambda(n,p,\gamma,s,w_1,w_2,w_3,w_4)>0$ such that
$$\lambda(\abs{D\ue}^2+\ez)^{\frac{p-2+s}{2}}\abs{D^2\ue}^2\leq S$$
almost everywhere in space in $\Omega_T$.
\end{lemma}

\begin{lemma}\label{lem:lemma-for-final-estimate}
Let $\ue:\Omega_T\to\R$ be a smooth solution to \eqref{eq:regularized-PDE}, and $S$ as in \eqref{eq:S}. Suppose that we can find weights $w_1,w_2,w_3,w_4\in\R$ such that 
$$\lambda(\abs{D\ue}^2+\ez)^{\frac{p-2+s}{2}}\abs{D^2\ue}^2\leq S$$
almost everywhere in space in $\Omega_T$, for some constant $\lambda=\lambda(n,p,\gamma,s,w_1,w_2,w_3,w_4)>0$. If $s\neq \gamma-p$, then for any concentric parabolic cylinders $Q_r\subset Q_{2r}\Subset \Omega_T$ with the center point $(x_0,t_0)\in \Omega_T$, we have the estimate
\begin{equation*}
    \begin{aligned}
        &\int_{Q_r}\abs{D\big((\abs{D\ue}^2+\ez)^{\frac{p-2+s}{4}}D\ue\big)}^2 dxdt\\
        \leq&\frac{C}{r^2}\Big(\int_{Q_{2r}}(\abs{D\ue}^2+\ez)^{\frac{p-2+s}{2}}\abs{D\ue}^2 dxdt+\int_{Q_{2r}}(\abs{D\ue}^2+\ez)^{\frac{p +s-\gamma}{2}} dxdt\Big)\\
        &+C\ez\Big(\frac{1}{r^2}\int_{Q_{2r}}\abs{\ln(\abs{D\ue}^2+\ez)} dxdt+\int_{B_{2r}}\abs{\ln(\abs{D\ue(x,t_0)}^2+\ez)} dx\Big)
    \end{aligned}
\end{equation*}
where $C=C(n,p,\gamma,s,w_1,w_2,w_3,w_4)>0.$
\end{lemma}
\section{Proof of main results}\label{sec:ProofofMain}
In this section, we will give the proof of the main results. The first part gives the uniform estimates for smooth solutions in the regularized case, and in the second part we pass to the original equation by letting $\ez\to 0$. We also consider time derivatives in Corollary \ref{cor:Time-derivative}.

\begin{proposition}\label{prop:LargePinPlane}
Let $\ue:\Omega_T\to\R$ be a smooth solution to above equation \eqref{eq:regularized-PDE}. If 
$$3\leq p\leq40\quad\text{and}\quad -1<\gamma<1,$$
then for any concentric parabolic cylinders $Q_r\subset Q_{2r}\Subset\Omega_T$ with the center point $(x_0,t_0)\in\Omega_T$, one has
\begin{equation}\label{eq:LargePinPlane-w22-estimate}
    \begin{aligned}
    \int_{Q_r}\abs{D^2\ue}^2dxdt
    \leq
    &\frac{C}{r^2}\Big(\int_{Q_{2r}}\abs{D\ue}^2dxdt+\int_{Q_{2r}}(\abs{D\ue}^2+\ez)^{\frac{2-\gamma}{2}}dxdt\Big)\\
    &+C\ez\Big(\frac{1}{r^2}\int_{Q_{2r}}\abs{\ln(\abs{D\ue}^2+\ez)}dxdt+\int_{B_{2r}}\abs{\ln\big(\abs{D\ue(x,t_0)}^2+\ez\big)} dx\Big),
\end{aligned}
\end{equation}
where $C=C(p,\gamma)>0$.
\end{proposition}

As explained in Section \ref{sec:steps} and in Lemma \ref{lem:positive-terms-away}, the key point to obtain the above estimate is choose suitable weights $w_1, w_2, w_3$ and $w_4\in\R$ to show that the quadratic form of auxiliary terms becomes positive definite with uniform constants. As now $s=2-p$, the weighted sum in (\ref{eq:S}) becomes
\begin{equation}\label{eq:S-sequal2minusp}
    \begin{aligned}
    S
    =&w_1GD_1(0)+w_2GD_2(0)+w_3\ez GD_1(-2)+w_4\ez GD_2(-2).
\end{aligned}
\end{equation}

\begin{lemma}\label{lem:ImprovedP-W22}
    Let $S$ be as in \eqref{eq:S-sequal2minusp}. Assume that
$$3\leq p\leq40 \quad \text{and}\quad -1<\gamma<1.$$
If we set
\begin{align}
\label{eq:w22-set}
        \begin{cases}
        w_1=p-\gamma &w_2=2 \\
        w_3=1-p &w_4=2(\sqrt{2}-1)
    \end{cases}
\end{align}
then 
$$ c\abs{D_T\abs{D\ue}}^2+Q\le S$$
where $c=c(p,\gamma)>0$ and 
$$Q=\left<\bar{x},M\bar{x}\right>$$
is a quadratic form with vector $\bar{x}=(\Delta_T \ue,\ilN \ue)^T\in\R^2$ and uniformly bounded positive definite (with a uniform constant) symmetric matrix $M=M(p,\gamma)\in \R^{2\times2}$.
\end{lemma}

The reader might wonder, how the above values of $w_1,w_2,w_3$ and $w_4$ were found. This was established by plugging in the parameters $a$ and $b$ in the proof below (and this is also the reason why we use them instead of their explicit values), and writing down the conditions (\ref{cond:w22}) in the proof. Then we select the parameter values so that all the conditions are met. In particular, choosing parameters such that $\det(M)\ge c>0$ requires some work including computing partial derivatives of $\det(M)$ with respect $a$ and $b$. Nonetheless, now that the weights are selected, checking that they satisfy the required conditions is quite straightforward.

\begin{proof}[Proof of Lemma \ref{lem:ImprovedP-W22}]
To prove the Lemma \ref{lem:ImprovedP-W22}, first, the coefficient of $\abs{D_T\abs{D\ue}}^2$ in \eqref{eq:S-expression-general} needs to be bounded from below by a positive constant, that is, 
\begin{align*}
    2c_1+c_2
    =&2w_1+2w_3\kappa^2\geq c
\end{align*}
uniformly in $\Omega_T$. 
Next, we should check that the matrix $M$ of the quadratic form $Q$ is uniformly bounded and uniformly positive definite. The uniform boundedness is easily obtained, and we can focus our attention to the uniform positive definiteness of $Q$. By Sylvester's criterion, it is equivalent to ensure the upper left 1-by-1 corner of $M$ and $M$ itself have a positive determinant, i.e.
\begin{align*} 
    c_3=w_2+w_4\kappa\geq c
\end{align*}
and
\begin{align*}
    \det(M)=c_3(c_3+c_4)P_\theta-\frac{1}{4}\big(c_3P_\theta+(c_3+c_4)-(2c_1+c_2)\big)^2
    \geq c
\end{align*}
uniformly in $\Omega_T$.
Collecting, we need to ensure that 
\begin{align}\label{cond:w22}
    \begin{cases}
        2c_1+c_2&=2(w_1+w_3\kappa^2)\geq c\\
        c_3&=w_2+w_4\kappa\geq c\\
        \det(M)&=c_3(c_3+c_4)P_\theta-\frac{1}{4}\big(c_3P_\theta+(c_3+c_4)-(2c_1+c_2)\big)^2
    \geq c
    \end{cases}
\end{align}
uniformly in $\Omega_T$.
Plugging values \eqref{eq:w22-set} into the conditions \eqref{cond:w22} with short hands $a=1-\gamma$ and $b=2\sqrt{2}$, the left of the first inequality can be rewritten as
\begin{align*}
    2c_1+c_2=2\big(p-\gamma+(a-p+\gamma)\kappa^2\big).
\end{align*}
Here we notice that the partial derivative  with respect to $\kappa$, that is $4(a-p+\gamma)\kappa$, has fixed sign because of the positivity of $\kappa$, then $2c_1+c_2$ is monotone to $\kappa$, which means the uniform lower bound of $2c_1+c_2$ can be achieved for the end points $\kappa=0$ or $\kappa=1$. 
Hence, we have
$$2c_1+c_2\geq\min\{2(p-\gamma), 2a\}>0$$
since $a>0$.
For the second condition in \eqref{cond:w22}, by the monotonicity of linear function with respect to $\kappa$, one has
\begin{align*}
    c_3=2+(b-2)\kappa\geq\min\{2, b\}>0
\end{align*}
since $b>0$.

For the last restriction, we need to check the uniform positivity of
\begin{align*}
     \det(M)
    =&\big(2+(b-2)\kappa\big)\Big(2(1-\gamma)+\big(2\gamma-(1+\gamma)(b-2)\big) \kappa+(2+\gamma)(b-2)\kappa^2\Big)\cdot\\
    &\big((p-2)(1-\kappa)+1\big)-\frac14 \kappa^2\big((b-4)(p-2-\gamma)(1-\kappa)+2(b-a)\kappa \big)^2\\
    =&:f(\kappa,p, \gamma)
\end{align*}
in $\Omega_T$.
Computing the first partial derivatives of the function $f$ with respect to $\gamma$, we have
\begin{equation*}
         \begin{aligned}
             f_\gamma(\kappa,\gamma,p)
             =&-(1-\kappa)\big((p-2)(1-\kappa)+1\big)\big((b-2)\kappa+2\big)^2\\
        &+\frac{1}{2}\kappa^2\big((1-\kappa)(b-4)-2\kappa\big)\big((b-4)(p-2-\gamma)(1-\kappa)+2(b-a)\kappa \big)\\
        =:&A\gamma+B,
         \end{aligned}
\end{equation*}
where $A=A(\kappa)$ and $B=B(\kappa, p)$ do not depend on $\gamma$ and 
\begin{equation*}
   A=-\frac{1}{2}\kappa^2\big(2\kappa+(4-b)(1-\kappa)\big)^2\leq 0.
\end{equation*}     
It remains to check the zeros of the gradient and the boundary values. At the boundary, this is just a polynomial of one variable, and again it suffices to check zeros of the derivatives and the endpoints. Instead of presenting these straightforward and lengthy computations, we summarize and illustrate the results in Figure \ref{fig1}. 

Since $ f_\gamma(\kappa,p,\gamma)\le 0$ for any admissible values of $\kappa$ and $p$ as explained above, it remains to check that $f(\kappa,p,1)$ is strictly positive. To check this positivity,  it would remain to repeat the same straightforward but lengthy computations as just explained for   $f_\gamma(\kappa,p,\gamma)$. Again, we simply summarize and illustrate the result in Figure~\ref{fig2}.
\begin{figure} 
\centering
\begin{subfigure}{.5\textwidth}
  \centering
\includegraphics[width=1.009\linewidth]{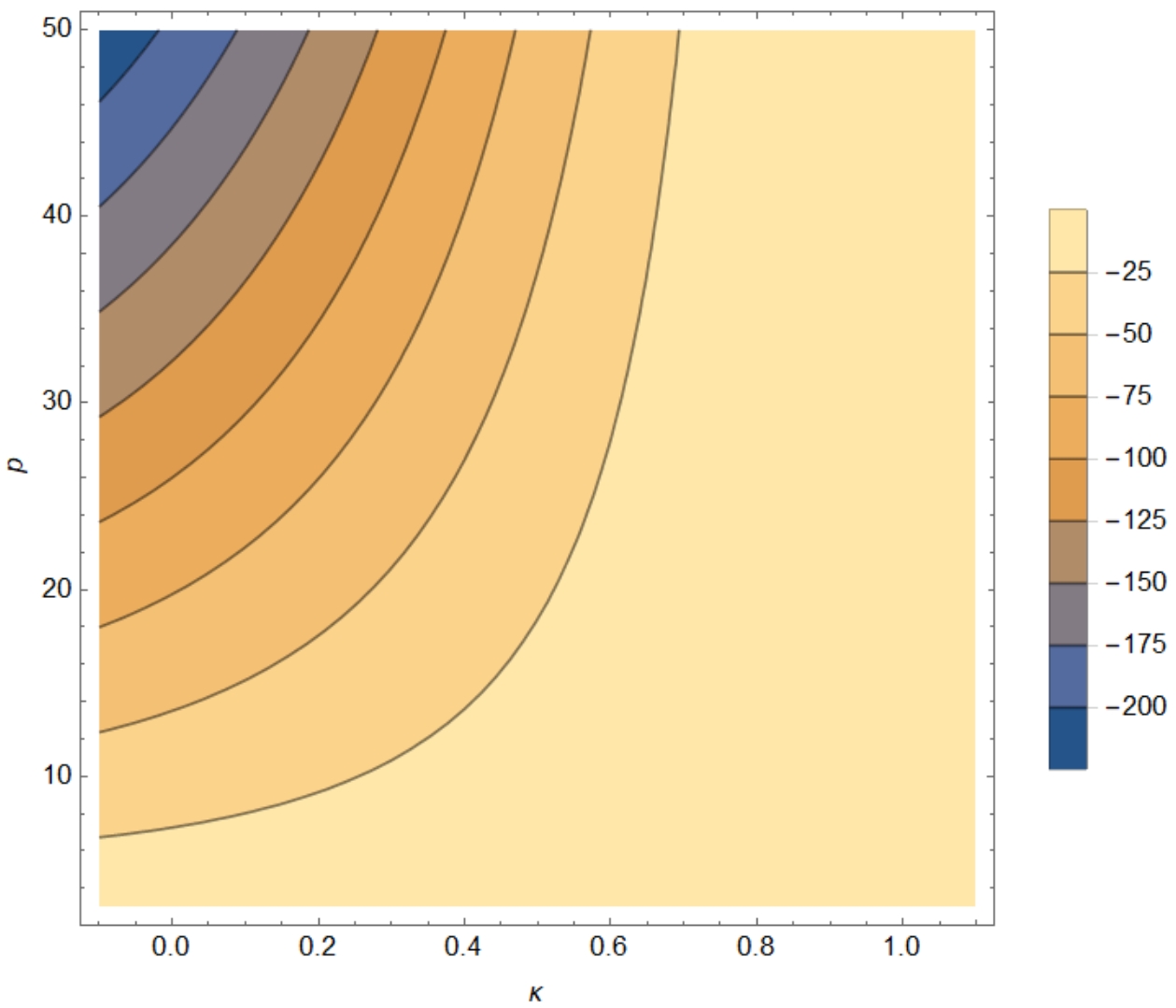}
    \caption{The value of $f_\gamma$ when $\gamma=-1$}
    \label{fig1}
\end{subfigure}%
\begin{subfigure}{.5\textwidth}
  \centering
\includegraphics[width=1.009\linewidth]{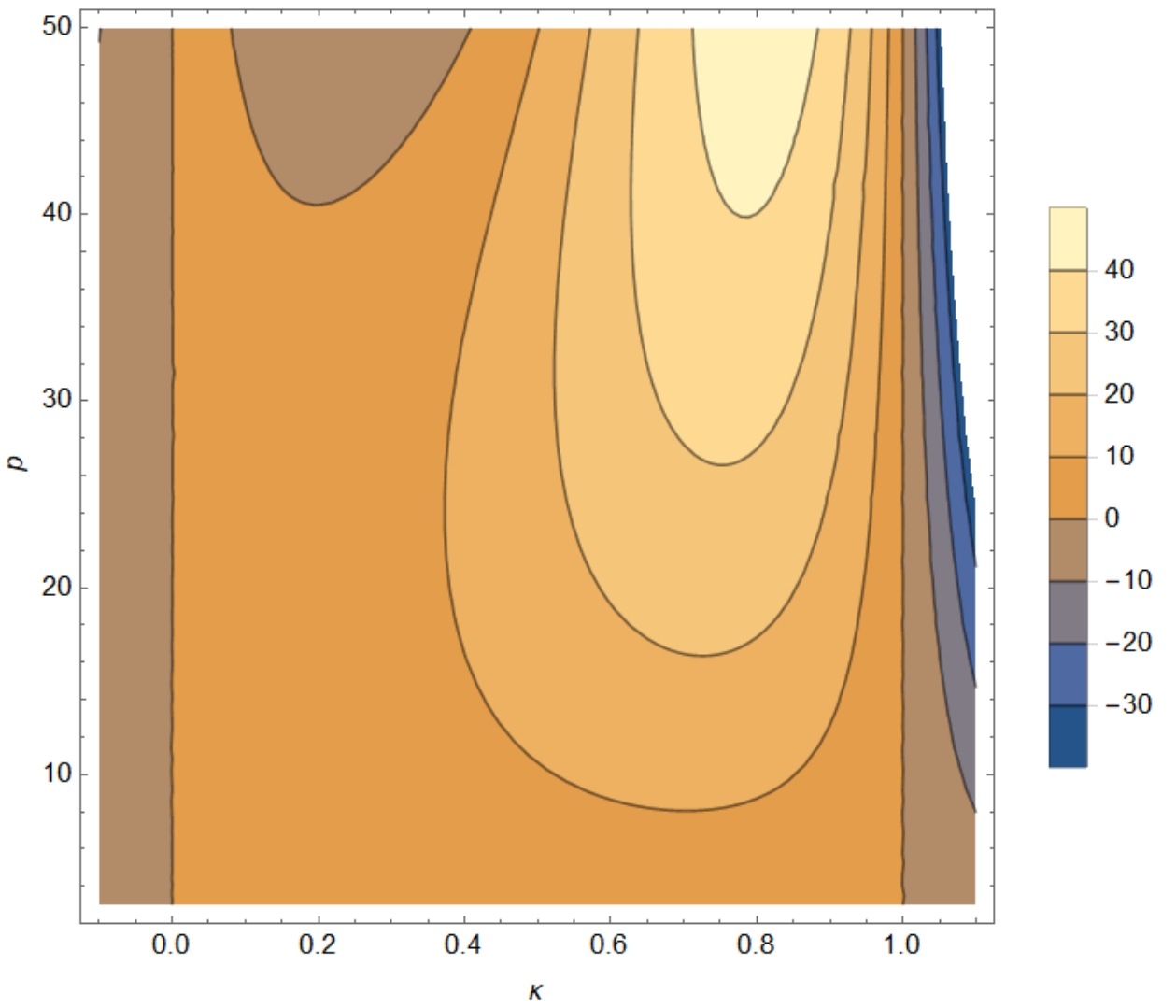}
    \caption{The value of $f$ when $\gamma=1$}
    \label{fig2}
\end{subfigure}
\caption{}
\label{fig:times}
\end{figure}
\end{proof}

Now we can give out the proof of Proposition \ref{prop:LargePinPlane}.
\begin{proof}[Proof of Proposition \ref{prop:LargePinPlane}]
    By the Lemma \ref{lem:ImprovedP-W22}, the conditions of  Lemma \ref{lem:positive-terms-away} are satisfied. Thus Lemma \ref{lem:lemma-for-final-estimate} implies the claim.
\end{proof}

\begin{remark}
The range of $p$ in Theorem \ref{thm:W22-more-p} can be slightly improved. Indeed, the key point in the proof of this theorem is the positive definiteness of a quadratic form $Q$. There are many choices of the weights $w_1,w_2,w_3$ and $w_4$ satisfying this condition.
However, for example by plotting the image of $\det(M)$ with a suitable fixed value of $\gamma$ as a function of $p$ and $\theta$, one can see that in the setting of Lemma \ref{lem:ImprovedP-W22} $\det(M)$ is negative for large values of $p$. Thus some upper bound for $p$ is needed there. 
\end{remark}

For the general $s$, we also have some results. To prove Theorem \ref{thm:W12-Nonlinear-Gradient-S}, we need the following proposition.

\begin{proposition}\label{prop:GeneralSinPlane}
Let $\ue:\Omega_T\to\R$ be a smooth solution to equation \eqref{eq:regularized-PDE}. If 
$$s>\max\{\gamma+1-p, -2-\gamma \}$$
or
$$-2-\gamma\geq s>\max\{\gamma+1-p, 2p-4-\gamma-2\sqrt{2(p-1)(p-2-\gamma)}\},$$
then for any concentric parabolic cylinders $Q_r\subset Q_{2r}\Subset\Omega_T$ with the center point $(x_0,t_0)\in\Omega_T$, one has
\begin{equation*}
    \begin{aligned} \int_{Q_r}&\abs{D\big((|D\ue|^2+\epsilon)^{\frac{p-2+s}{4}}D\ue\big)}^2dxdt\\
    \leq&
    \frac{C}{r^2}\Big(
    \int_{Q_{2r}}(|D\ue|^2+\epsilon)^{\frac{p-2+s}{2}}|D\ue|^2dxdt 
    +\int_{Q_{2r}}(|D\ue|^2+\epsilon)^{\frac{p+s-\gamma}{2}}dxdt\Big) \\
    &+C\epsilon\Big(\frac{1}{r^2}\int_{Q_{2r}}\big|\ln(|D\ue|^2+\epsilon)\big|dxdt 
    +\int_{B_{2r}}\big|\ln(|D\ue(x,t_0)|^2+\epsilon)\big|dx\Big),
    \end{aligned}
\end{equation*} 
where $C=C(p,\gamma,s)>0$.
\end{proposition}

\begin{lemma}\label{lem:General-s}
Let $S$ be as in above \eqref{eq:S}. Assume that
$$s>\max\{\gamma+1-p, -2-\gamma\}$$
or
$$-2-\gamma\geq s>\max\{\gamma+1-p, 2p-4-\gamma-2\sqrt{2(p-1)(p-2-\gamma)}\}.$$
For small enough $\eta=\eta(p,\gamma,s)>0$, if we set
\begin{align*}
    \begin{cases}
        w_1=2p-2+s-\gamma-2\sqrt{(p-1) (p-1+s-\gamma)}+\eta  &w_2=p+s  \\
        w_3=0  &w_4=0 
    \end{cases}
\end{align*}
then 
$$(\abs{D\ue}^2+\epsilon)^{\frac{p-2+s}{2}}\left\{c\abs{D_T\abs{D\ue}}^2+Q\right\}\le S,$$
where $c=c(p,\gamma)>0$ and 
$$Q=\left<\bar{x},M\bar{x}\right>$$
is a quadratic form with vector $\bar{x}=(\Delta_T \ue,\ilN \ue)^T\in\R^2$ and uniformly bounded positive definite (with a uniform constant) symmetric matrix $M=M(p,\gamma)\in \R^{2\times2}$.
\end{lemma}
\begin{proof}
Plugging $w_3=w_4=0$ into the values \eqref{eq:expression-c1-c4} of $c_1$ to $c_4$, one has
\begin{align*} 
   \begin{cases}
        c_1=w_1,\quad &c_2=w_1(p-2+s)\theta,\\
        c_3=w_2,\quad 
        &c_4=w_2(p-2+s-\gamma)\theta.
    \end{cases} 
\end{align*}
Then we rewrite the weighted sum \eqref{eq:S-expression-general} as
\begin{equation*} 
    \begin{aligned}
    & (\abs{D\ue}^2+\epsilon)^{\frac{p-2+s}{2}}\Big\{w_1(P_\theta+S_\theta) \abs{D_T\abs{D\ue}}^2+w_2(\Delta_T \ue)^2\\
    &+w_2(P_\theta+S_\theta-K_\theta)P_\theta(\ilN \ue)^2+\big(w_2(2P_\theta+S_\theta-K_\theta)-w_1(P_\theta+S_\theta)\big)\Delta_T \ue\ilN \ue\Big\}=S,
    \end{aligned}
\end{equation*}
where we denote 
$$S_\theta:=1+s\theta\quad\text{and}\quad K_\theta=
1+\gamma\theta\in(0,\infty)$$
for sake of brevity. 
We fix $w_2=p+s$ whilst leaving $w_1$ as a free variable.  We have
\begin{equation*} 
    \begin{aligned}
    &  (\abs{D\ue}^2+\epsilon)^{\frac{p-2+s}{2}}\Big\{w_1( P_\theta+S_\theta  
   )\abs{D_T\abs{D\ue}}^2+(p+s)(\Delta_T \ue)^2\\
    &+(p+s) (P_\theta+S_\theta-K_\theta)P_\theta(\ilN \ue)^2+\big((p+s) (2P_\theta+S_\theta-K_\theta)-w_1 ( P_\theta+S_\theta) \big)\Delta_T \ue\ilN \ue\Big\}\\
    =&(\abs{D\ue}^2+\epsilon)^{\frac{p-2+s}{2}}\big\{w_1( P_\theta+S_\theta  
    )\abs{D_T\abs{D\ue}}^2+Q\big\}=S
    \end{aligned}
    \end{equation*}
where the quadratic form $Q$ depends on $\Delta_T \ue$ and $\ilN \ue $. 
The matrix of $Q$ is 
$$M(\theta):=
\begin{bmatrix}
    p+s  &\displaystyle\frac12 \Big( (p+s)(2P_\theta+S_\theta-K_\theta)-w_1 (P_\theta+S_\theta  )\Big)\\
    \displaystyle\frac12 \Big((p+s) (2P_\theta+S_\theta-K_\theta)-w_1(P_\theta+S_\theta)\Big) & (p+s)(P_\theta+S_\theta-K_\theta)P_\theta
\end{bmatrix}.
$$
We will check the condition similar to (\ref{cond:w22}).
First we observe that $w_2=p+s>\gamma+1>0$ since the range of $s$ shows that $s>\gamma+1-p$. It holds that 
\begin{equation}\label{eq:positivity-P_theta+S_theta}
\begin{aligned}
    P_\theta+S_\theta=&(p-2+s)\theta+2\\
    >& P_\theta+S_\theta-K_\theta=(p-2+s-\gamma)\theta+1\\
    \geq& \min\{p-1+s-\gamma,1\}>0.
\end{aligned}
\end{equation} 
In order to conclude that the first condition similar to the first condition in \eqref{cond:w22} is satisfied, that is, coefficient of $|D_T|D\ue||^2$ is positive, we also need to ensure that $w_1>0$. This is clear, because eventually in (\ref{eq:choice-w1}) we can choose $w_1>0$.
Then the first two conditions of \eqref{cond:w22} are satisfied. 

The third, the determinant of $Q$ is uniformly positive if and only if 
$$(p+s)^2(P_\theta+S_\theta-K_\theta)P_\theta-\frac14\Big((p+s)
(2P_\theta+S_\theta-K_\theta)-w_1 ( P_\theta+S_\theta  )\Big)^2\geq c>0$$
for all $\theta\in[0,1)$.
We interpret this expression as a second order polynomial in $w_1$. Let $w_1^-:=w_1^-(\theta)$ and $w_1^+:=w_1^+(\theta)$ denote the roots of such polynomial. 
We conclude that $\operatorname{det}(M(\theta))$ is uniformly positive if and only if
\begin{equation*}
    \begin{aligned}
        w_1^+(\theta):=&\frac{p+s}{P_\theta+S_\theta}\big(\sqrt{P_\theta+S_\theta-K_\theta}+\sqrt{P_\theta}\big)^2\\
        >&w_1\\
        >&\frac{p+s}{P_\theta+S_\theta}\big(\sqrt{P_\theta+S_\theta-K_\theta}-\sqrt{P_\theta}\big)^2:=w_1^-(\theta)
    \end{aligned}
\end{equation*}
uniformly in $\Omega_T$. Thus we may select a weight $w_1$ such that  the determinant is always positive if 
$$\inf_\theta w_1^+(\theta)>\sup_\theta w_1^-(\theta).$$
Thus we need to look at the extremals of $w_1^-$ and $w_1^+$.
Computing the derivative of $w_1^-$, we have
    \begin{align*}
        (w_1^-)^{'}(\theta)
        =&\frac{(p+s)\big(\sqrt{P_\theta+S_\theta-K_\theta}-\sqrt{P_\theta}\big)}{\big(P_\theta+S_\theta\big)^2\sqrt{P_\theta+S_\theta-K_\theta}\sqrt{P_\theta}}\cdot\\
        &\Big\{(p-2+s-2\gamma)\sqrt{P_\theta}-
        (p-2-s)\sqrt{P_\theta+S_\theta-K_\theta}\Big\}.
    \end{align*}
Thus $(w_1^-)^{'}(\theta) = 0$ if and only if either $s=\gamma$ or
$$(p-2+s-2\gamma)\sqrt{P_\theta}=
        (p-2-s)\sqrt{P_\theta+S_\theta-K_\theta}$$
where $\theta\in(0,1)$. Here we omit the case $\theta=0$, since it is a boundary point where we should consider the value of function $w_1^-$.  
We obtain a stationary point
\begin{align*}
    \theta_1
    =&\frac{4(p-2-\gamma)}{(p-2-s)^2-4(p-2-\gamma)(p-2)}
\end{align*} 
if $(p-2+s-2\gamma)(p-2-s)>0$ but this will lead to a contradiction. 
Indeed, then
\begin{align*}
    P_{\theta_1}
    =&\frac{(p-2-s)^2}{(p-2-s)^2-4(p-2-\gamma)(p-2)},
\end{align*}
and 
\begin{align*}
    P_{\theta_1}+S_{\theta_1}
    =&\frac{-2(p-2+s-2\gamma)(p-2-s)}{(p-2-s)^2-4(p-2-\gamma)(p-2)}
\end{align*}
have obviously different signs even if both should be positive by the earlier assumptions \eqref{def:P_theta} and \eqref{eq:positivity-P_theta+S_theta}.
Hence, we can obtain the supremum of $w_1^-$ with respect to $\theta$ by considering the endpoints as
\begin{align*}
    \sup_\theta w_1^-(\theta)=&\max\{w_1^-(0), w_1^-(1)\}\\
    =&\max\{0 ,\big(\sqrt{p-1+s-\gamma}-\sqrt{p-1}\big)^2\}\\
    =&\big(\sqrt{p-1+s-\gamma}-\sqrt{p-1}\big)^2.
\end{align*}

Similarly, we calculate the derivative of $w_1^+(\theta)$:
\begin{align*}
        (w_1^+)^{'}(\theta)
        =&\frac{(p+s)\big(\sqrt{P_\theta+S_\theta-K_\theta}+\sqrt{P_\theta}\big)}{\big(P_\theta+S_\theta\big)^2\sqrt{P_\theta+S_\theta-K_\theta}\sqrt{P_\theta}}\cdot\\
        &\Big\{(p-2+s-2\gamma)\sqrt{P_\theta}+
        (p-2-s)\sqrt{P_\theta+S_\theta-K_\theta}\Big\},
\end{align*}
and from $w_1^+(\theta_2)=0$, we may solve
\begin{align*}
    \theta_2 
    =&\frac{4(p-2-\gamma)}{(p-2-s)^2-4(p-2-\gamma)(p-2)}
\end{align*}
if $(p-2+s-2\gamma)(p-2-s)<0$. Otherwise if $(p-2+s-2\gamma)(p-2-s)\geq0$, a straightforward computation shows that  
$$(p-2+s-2\gamma)\sqrt{P_\theta}+(p-2-s)\sqrt{P_\theta+S_\theta-K_\theta}$$
does not vanish except if $p-2=s=\gamma$. 
By a direct computation
\begin{align*}
     w_1^+(\theta_2)
    =&(p+s)\Big(\frac{P_{\theta_2}+S_{\theta_2}-K_{\theta_2}}{P_{\theta_2}+S_{\theta_2}}+\frac{P_{\theta_2}}{P_{\theta_2}+S_{\theta_2}}+\sqrt{\frac{P_{\theta_2}+S_{\theta_2}-K_{\theta_2}}{P_{\theta_2}+S_{\theta_2}}\frac{P_{\theta_2}}{P_{\theta_2}+S_{\theta_2}}}\Big)\\
    =&\frac{p+s}{2}\Big(-\frac{p-2+s-2\gamma}{p-2-s}-\frac{p-2-s}{p-2+s-2\gamma}+2\Big)\\
    \geq& 2(p+s).
\end{align*}
Hence,
\begin{align*}
   \inf_\theta w_1^+(\theta)=&\min\{w_1^+(0),w_1^+(\theta_2), w_1^+(1)\}\\
   =&\min\{2(p+s),w_1^+(\theta_2),\big(\sqrt{p-1+s-\gamma}+\sqrt{p-1}\big)^2\}\\
   =&\min\{2(p+s), \big(\sqrt{p-1+s-\gamma}+\sqrt{p-1}\big)^2\}.
\end{align*}
Comparing the formulas for $\inf_\theta w_1^+(\theta)$ and $\sup_\theta w_1^-(\theta)$, we see that
$$\inf_\theta w_1^+(\theta)>\sup_\theta w_1^-(\theta)$$
if and only if 
\begin{align}\label{cond:RestrictionS}
    2(p+s)>\big(\sqrt{p-1+s-\gamma}-\sqrt{p-1}\big)^2
\end{align} 
under the original assumption $s>\gamma+1-p$. 
Next we consider two cases $s>-2-\gamma$ and $s\le -2-\gamma$. In the first case, by a direct computation, if
$$s>\max\{\gamma+1-p, -2-\gamma\},$$
then the condition \eqref{cond:RestrictionS} is satisfied. On the other hand, if $s\leq-2-\gamma$, we note that
$$p-2-\gamma\geq p+s>\gamma+1>0,$$
and  
$$2p-4-\gamma-2\sqrt{2(p-1)(p-2-\gamma)}<s<2p-4-\gamma+2\sqrt{2(p-1)(p-2-\gamma)}$$
from the inequality \eqref{cond:RestrictionS}.
Thus collecting the observations, we have
\begin{align*}
    \max&\{\gamma+1-p,2p-4-\gamma-2\sqrt{2(p-1)(p-2-\gamma)} \}\\
    &<s
    <\min\{2p-4-\gamma+2\sqrt{2(p-1)(p-2-\gamma)},-2-\gamma\} =-2-\gamma,
\end{align*}
and we also obtain the inequality \eqref{cond:RestrictionS}.
Hence if 
$$s>\max\{\gamma+1-p, -2-\gamma \}$$
or
$$-2-\gamma\geq s>\max\{\gamma+1-p, 2p-4-\gamma-2\sqrt{2(p-1)(p-2-\gamma)}\},$$
then \eqref{cond:RestrictionS} is satisfied.
Moreover, we can choose the value of $w_1$ as follows
\begin{align}
\label{eq:choice-w1}
    w_1=&\big(\sqrt{p-1+s-\gamma}-\sqrt{p-1}\big)^2+\eta \\
    =&2p-2+s-\gamma-2\sqrt{(p-1+s-\gamma)(p-1)}+\eta\nonumber 
\end{align} 
where $\eta>0$ small enough. 
\end{proof}

We conclude this subsection with the proof of Proposition \ref{prop:GeneralSinPlane}.

\begin{proof}[Proof of Proposition \ref{prop:GeneralSinPlane}.]
Lemma \ref{lem:General-s}, under the assumption 
$$s>\max\{\gamma+1-p, -2-\gamma \}$$
or
$$-2-\gamma\geq s>\max\{\gamma+1-p, 2p-4-\gamma-2\sqrt{2(p-1)(p-2-\gamma)}\},$$  
gives the key estimate needed for Lemma \ref{lem:positive-terms-away}. This allows us to use Lemma \ref{lem:lemma-for-final-estimate} to obtain the claim.
\end{proof}

\subsection{Passing to the original equation}
In this section, we remove the smoothness assumption and pass $\ez$ to $0$ to get the desired estimates, and thus to prove Theorem \ref{thm:W22-more-p} and Theorem \ref{thm:W12-Nonlinear-Gradient-S}.

\begin{proof}[Proof of Theorem \ref{thm:W22-more-p} and \ref{thm:W12-Nonlinear-Gradient-S}]
This is exactly similar to the final stage of proof of Theorem 1.1 in \cite{fengps23}. Thus we only recall that idea is to approximate a viscosity solution to  
    \begin{equation*}
        u_t-\abs{Du}^\gamma\big(\Delta u+(p-2)\ilN u\big)=0
    \end{equation*}
by
    \begin{equation*}
    \begin{aligned}
        \begin{cases}
            \ue_t-(\abs{D\ue}^2+\ez)^{\frac{\gamma}{2}}\Big(\Delta \ue+(p-2)\frac{\ilN \ue}{\abs{D\ue}^2+\ez}\Big)=0     \quad&{\rm in}\; U_{t_1,t_2};\\
            \ue=u   \quad& {\rm on}\;\partial_pU_{t_1,t_2}
        \end{cases}
    \end{aligned}   
    \end{equation*}
in a cylinder $U_{t_1,t_2}$. Here
$\partial_pU_{t_1,t_2}:=(U\times\{t_1\})\cup(\partial U\times [t_1,t_2])$ denotes 
the parabolic boundary. Hence, Proposition \ref{prop:LargePinPlane} and \ref{prop:GeneralSinPlane} are applicable to $\ue$ respectively, and the key steps read as 
\begin{align*}
    \int_{Q_{r}}|D^2u|^2dxdt
    &\leq \liminf_{\epsilon\to 0}\int_{Q_{r}}|D^2\ue|^2dxdt \\
    &\leq \liminf_{\epsilon\to 0} \Bigg(\frac{C}{r^2}\Big(\int_{Q_{2r}} |D\ue|^2 dxdt 
    +\int_{Q_{2r}}(|D\ue|^2+\epsilon)^{\frac{2-\gamma}{2}}dxdt\Big) \\
    &\quad
    +C\epsilon \Big(\frac{1}{r^2}\int_{Q_{2r}}\big|\ln(|D\ue|^2+\epsilon)\big|dxdt
    +\int_{B_{2r}}\big|\ln\big(|D\ue(x,t_0)|^2+\epsilon\big)\big|dx
    \Big)\Bigg) \\
    &= \frac{C}{r^2}\Big(\int_{Q_{2r}}|Du|^2 dxdt 
    +\int_{Q_{2r}}|Du|^{2-\gamma} dxdt\Big),
\end{align*}
and
\begin{align*}
&\int_{Q_r}\abs{D\big(\abs{Du}^{\frac{p-2+s}{2}}D\ue\big)}^2dxdt\\    \leq &\liminf_{\ez\to 0} \int_{Q_r} \abs{D\big((|D\ue|^2+\epsilon)^{\frac{p-2+s}{4}}D\ue\big)}^2dxdt\\
    \leq&\liminf_{\ez\to 0} \Bigg(
    \frac{C}{r^2} \Big(
    \int_{Q_{2r}}(|D\ue|^2+\epsilon)^{\frac{p-2+s}{2}}|D\ue|^2dxdt 
    +\int_{Q_{2r}}(|D\ue|^2+\epsilon)^{\frac{p+s-\gamma}{2}}dxdt\Big) \\
    &+C\epsilon\Big(\frac{1}{r^2}\int_{Q_{2r}}\big|\ln(|D\ue|^2+\epsilon)\big|dxdt 
    +\int_{B_{2r}}\big|\ln(|D\ue(x,t_0)|^2+\epsilon)\big|dx\Big)\Bigg)\\
    = &\frac{C}{r^2} \Big(
    \int_{Q_{2r}}\abs{Du}^{p+s} dxdt 
    +\int_{Q_{2r}}\abs{Du}^{p+s-\gamma}dxdt\Big),
    \end{align*}
giving the desired estimates. Here we utilized the uniform bounds on $D\ue$, due to \cite{imbertjs19}.
\end{proof}

Finally, we prove the corollary about time derivative.

\begin{proof}[Proof of Corollary \ref{cor:Time-derivative}]
We utilize the same idea as in \cite[Corollary 1.2]{fengps23}.  
Let $\ue$ be a smooth solution to \eqref{eq:regularized-PDE}. 
We aim to show that $\{\ue_t\}_{\epsilon}$ is uniformly bounded in $L^2_{\loc}(\Om_T)$.

We first focus on range (i), that is, $3\leq p\leq 40$ and $0\leq \gamma<1$. 
We employ the pointwise estimate
\begin{equation}\label{eq:ut-upper-bound}
    \begin{aligned}
        \abs{\ue_t}=&\abs{(|D\ue|^2+\epsilon)^{\frac{\gamma}{2}}\Big(\Delta\ue+(p-2)\frac{\il\ue}{|D\ue|^2+\epsilon}\Big)}\\
        \leq&(|D\ue|^2+\epsilon)^{\frac{\gamma}{2}}\big(\abs{\Delta\ue}+\abs{p-2}\abs{D^2\ue}\big)\\
        \leq&(p+2)(|D\ue|^2+\epsilon)^{\frac{\gamma}{2}} \abs{D^2\ue}
    \end{aligned}
\end{equation} 
and conclude that
\begin{equation*}
    \begin{aligned}
        \int_{Q_r}\abs{\ue_t}^2dxdt
        \leq (p+2)^2(\|D\ue\|_{L^{\infty}(Q_r)}^2+\epsilon)^\gamma\int_{Q_r}\abs{D^2\ue}^2dxdt.
    \end{aligned}
\end{equation*}
Under the assumptions $3\leq p\leq 40$ and $0\leq \gamma<1$ we can employ Proposition \ref{prop:LargePinPlane} and the uniform bound of $\|D\ue\|_{L^\infty(Q)}$ due to \cite{imbertjs19} to conclude that $\{\ue_t\}_{\epsilon}$ is indeed uniformly bounded in $L^2_{\loc}(\Om_T)$.

We next focus on the range (ii), $1<p<9\gamma+10$ and $-1<\gamma<\infty$.
We employ the pointwise estimate
\begin{align*}
    \abs{\ue_t}^2
    \leq&C(p) (|D\ue|^2+\epsilon)^{\gamma}\abs{D^2\ue}^2\\
    \leq&C(p,\gamma) \abs{D\big((\abs{D\ue}^2+\ez)^{\frac{\gamma}{2}}D\ue\big)}^2.
\end{align*}
and conclude that
\begin{equation*}
    \begin{aligned}
        \int_{Q_r}\abs{\ue_t}^2dxdt
        \leq C(p,\gamma)\int_{Q_r} \abs{D\big((\abs{D\ue}^2+\ez)^{\frac{\gamma}{2}}D\ue\big)}^2dxdt.
    \end{aligned}
\end{equation*}
Under the assumptions  $1<p<9\gamma+10$ and $-1<\gamma<\infty$ we can employ Proposition \ref{prop:GeneralSinPlane} with $s=2\gamma+2-p$ and the uniform bound of $\|D\ue\|_{L^\infty(Q)}$ due to \cite{imbertjs19} to conclude that $\{\ue_t\}_{\epsilon}$ is indeed uniformly bounded in $L^2_{\loc}(\Om_T)$.

Given the uniform boundedness of $\{\ue_t\}_{\epsilon}$ in $L^2_{\loc}(\Om_T)$, we employ the weak compactness of Lebesgue spaces and the uniform convergence $\ue\to u$ of viscosity solutions to find that (up to some subsequence)
\begin{align*}
    \int_{Q_r}|u_t|^2dxdt \leq \liminf_{\epsilon\to 0}\int_{Q_r}|\ue_t|^2dxdt \leq C
\end{align*}
where $C>0$ is independent of $\epsilon$. The proof is finished.
\end{proof}


\begin{thebibliography}{10}

\bibitem{andrades22}
P.~D.~S. Andrade and M.~S. Santos.
\newblock Improved regularity for the parabolic normalized {$p$}-{L}aplace
  equation.
\newblock {\em Calc.\ Var.\ Partial Differential Equations}, 61(5):Paper No. 196,
  13, 2022.

\bibitem{attouchi20}
A.~Attouchi.
\newblock Local regularity for quasi-linear parabolic equations in
  non-divergence form.
\newblock {\em Nonlinear Anal.}, 199:112051, 28, 2020.

\bibitem{attouchip18}
A.\ Attouchi and M.~Parviainen.
\newblock H\"older regularity for the gradient of the inhomogeneous parabolic
  normalized $p$-{L}aplacian.
\newblock {\em Commun.\ Contemp.\ Math.}, 20(4):1750035, 27, 2018.

\bibitem{attouchir20}
A.\ Attouchi and E.\ Ruosteenoja.
\newblock Gradient regularity for a singular parabolic equation in
  non-divergence form.
\newblock {\em Discrete Contin.\ Dyn.\ Syst.}, 40(10):5955--5972, 2020.

\bibitem{cianchim19}
A.~Cianchi and V.~G. Maz'ya.
\newblock Second-order regularity for parabolic $p$-Laplace problems.
\newblock {\em J.\ Geom.\ Anal.}, 30(2):1565--1583, 2020.

\bibitem{cordes61}
H.~O.~Cordes.
\newblock Zero order a priori estimates for solutions of elliptic differential equations.
\newblock {\em Proc.\ {S}ympos.\ {P}ure {M}ath.}, 4:157--166,1961.

\bibitem{dibenedetto93}
E.~DiBenedetto.
\newblock {\em Degenerate parabolic equations}.
\newblock Universitext. Springer-Verlag, New York, 1993.

\bibitem{does11}
K.~Does.
\newblock An evolution equation involving the normalized $p$-{L}aplacian.
\newblock {\em Commun.\ Pure Appl.\ Anal.}, 
10(1):361--396, 2011.

\bibitem{dongpzz20}
H.~Dong, F.~Peng, Y.~Zhang, and Y.~Zhou.
\newblock Hessian estimates for equations involving {$p$}-{L}aplacian via a
  fundamental inequality.
\newblock {\em Adv.\ Math.}, 370:107212, 40, 2020.

\bibitem{fengps22}
Y.~Feng, M.~Parviainen, S.~Sarsa.
\newblock On the second-order regularity of solutions to the parabolic $p$-Laplace equation.
\newblock {\em J.\ Evol.\ Equ.}, 22(1):Paper No.6, 17, 2022.

\bibitem{fengps23}
Y.~Feng, M.~Parviainen, S.~Sarsa.
\newblock A systematic approach on the second order regularity of solutions to the general parabolic $p$-Laplace equation.
\newblock {\em Calc.\ Var.\ Partial Differential Equations.}, 62(7):Paper No.204, 39, 2023.

\bibitem{giga06}
Y. Giga.
\newblock  Surface evolution equations: A level set approach.
\newblock {\em Monographs in Mathematics, 99.} Birkh\"{a}user Verlag, 2006.

\bibitem{lindqvisth20}
F.~A.~H{\o}eg and P.~Lindqvist.
\newblock Regularity of solutions of the parabolic normalized $p$-Laplace equation.
\newblock {\em Adv.\ Nonlinear Anal.}, 9(1):7--15, 2020.

\bibitem{imbertjs19}
C.~Imbert, T.~Jin, and L.~Silvestre.
\newblock H\"{o}lder gradient estimates for a class of singular or degenerate parabolic equations.
\newblock {\em Adv.\ Nonlinear Anal.}, 8(1):845--867, 2019.

\bibitem{jins17}
T.~Jin and L.~Silvestre.
\newblock H\"older gradient estimates for parabolic homogeneous $p$-{L}aplacian
  equations.
\newblock {\em J.\ Math.\ Pures Appl.}, 108(1):63--87, 2017.

\bibitem{kurkinenps23}
T. Kurkinen, M. Parviainen, and J. Siltakoski.
\newblock Elliptic Harnack's inequality for a singular nonlinear parabolic equation in non-divergence form.
\newblock {\em Bull.\ Lond.\ Math.\ Soc.}, 55(1):470--489, 2023.

\bibitem{kurkinens}
T. Kurkinen and J. Siltakoski.
\newblock Intrinsic Harnack's inequality for a general nonlinear parabolic equation in non-divergence form.
\newblock {\em arXiv preprint}:2308.13443, 2023.

\bibitem{lindqvist08}
P.~Lindqvist.
\newblock On the time derivative in a quasilinear equation.
\newblock {\em Skr.\ K.\ Nor\ Vidensk.\ Selsk.}, (2):1--7, 2008.

\bibitem{lindqvist12}
P.~Lindqvist.
\newblock On the time derivative in an obstacle problem.
\newblock {\em Rev.\ Mat.\ Iberoam.}, 28(2): 577--590, 2012.

\bibitem{lindqvist17}
P.~Lindqvist.
\newblock The time derivative in a singular parabolic equation.
\newblock {\em Differential Integral Equations}, 30(9-10):795--808, 2017.

\bibitem{manfredipr10}
J.~J. Manfredi, M.~Parviainen, and J.~D. Rossi.
\newblock An asymptotic mean value characterization for $p$-harmonic functions.
\newblock {\em Proc.\ Amer.\ Math.\ Soc.}, 138(3):881--889, 2010.

\bibitem{manfredi88}
J.~J.~Manfredi and A.~Weitsman.
\newblock On the Fatou theorem for $p$-harmonic functions.
\newblock {\em Comm.\ Partial Differential Equations}, 13(6):651--668, 1988.

\bibitem{maugerips00}
A.~Maugeri, D.~K.~Palagachev, and L.~G.~Softova.
\newblock  Elliptic and Parabolic Equations with Discontinuous Coefficients.
\newblock {\em Mathematical Research, 109.} Wiley-VCH Verlag Berlin GmbH, Berlin, 2000.

\bibitem{ohnumas97}
M.~Ohnuma and K.~Sato.
\newblock Singular degenerate parabolic equations with applications to the
$p$-Laplace diffusion equation.
\newblock {\em Comm.\ Partial Differential Equations}, 22(3-4):381-411, 1997.

\bibitem{parviainenv20}
M.~Parviainen and J.~L. V{\'a}zquez.
\newblock Equivalence between radial solutions of different parabolic gradient-diffusion equations and applications.
\newblock {\em Ann.\ Sc.\ Norm.\ Super.\ Pisa Cl.\ Sci.}, 5(21):303--359,2020.

\bibitem{sarsa22}
S.~Sarsa.
\newblock Note on an elementary inequality and its application to the regularity of {$p$}-harmonic functions.
\newblock {\em Ann.\ Fenn.\ Math.}, 47(1):139--153, 2022.

\end{thebibliography}
\end{document}